\documentclass{amsart}
\usepackage{amssymb,stmaryrd}
\usepackage{hyperref}
\usepackage{xcolor}

\newtheorem{theorem}{Theorem}[section]
\newtheorem{corollary}[theorem]{Corollary}
\newtheorem{proposition}[theorem]{Proposition}
\newtheorem{lemma}[theorem]{Lemma}

\theoremstyle{definition}
\newtheorem{definition}[theorem]{Definition}

\DeclareMathOperator{\Cay}{Cay}
\DeclareMathOperator{\Sym}{Sym}
\DeclareMathOperator{\Alt}{Alt}
\DeclareMathOperator{\Aut}{Aut}
\DeclareMathOperator{\Out}{Out}

\DeclareMathOperator{\PSL}{PSL}
\DeclareMathOperator{\PSU}{PSU}
\DeclareMathOperator{\PGaL}{P\Gamma L}
\DeclareMathOperator{\Z}{Z}

\newcommand{\Cent}{\mathbf{C}}
\newcommand{\N}{\mathbf{N}}

\newcommand{\K}{\mathrm{K}}
\newcommand{\M}{\mathrm{M}}
\newcommand{\C}{\mathrm{C}}
\newcommand{\D}{\mathrm{D}}
\newcommand{\Q}{\mathrm{Q}}

\newcommand{\ZZ}{\mathbb Z}
\newcommand{\FF}{\mathbb F}

\begin{document}
\title[Squarefree GRR/DRR detection]{Detecting Graphical and Digraphical Regular Representations in groups of squarefree order}

\author{Joy Morris}           
\address{Department of Mathematics and Computer Science\\
University of Lethbridge\\
Lethbridge, AB. T1K 3M4}    
\email{joy.morris@uleth.ca}
\thanks{Supported by the Natural Science and Engineering Research Council of Canada (grant RGPIN-2017-04905).}

\author{Gabriel Verret}
\address{Department of Mathematics, University of Auckland, 38 Princes Street, 1010 Auckland, New Zealand}
\email{g.verret@auckland.ac.nz}

\keywords{Cayley graph, Cayley digraph, GRR, DRR, automorphism group, normaliser}

\begin{abstract}
A necessary condition for a Cayley digraph $\Cay(R,S)$ to be a regular representation is that there are no non-trivial group automorphisms of $R$ that fix $S$ setwise. A group is DRR-detecting or GRR-detecting  if this condition is also sufficient for all Cayley digraphs or graphs on the group, respectively. In this paper, we determine precisely which groups of squarefree order are DRR-detecting, and which are GRR-detecting.
\end{abstract}

\maketitle

\section{Introduction and background}
All groups and digraphs in this paper are finite. Given a group $R$ and a subset $S \subseteq R$, the \emph{Cayley digraph} $\Cay(R,S)$ is the digraph with vertex-set $R$, with an arc from $r$ to $sr$ if and only if $s \in S$. If $S=S^{-1}$ then we also say that $\Cay(R,S)$ is a \emph{Cayley graph}. It is straightforward to show that the right-regular representation $\hat{R}$ of $R$ is a subgroup of  the automorphism group of $\Cay(R,S)$. It is also not hard to show that, conversely, if the automorphism group of a digraph admits a regular subgroup isomorphic to $R$, then the digraph is isomorphic to $\Cay(R,S)$ for some $S \subseteq R$. Digraphs such that their (full) automorphism group is regular are of special interest.

\begin{definition}
A Cayley digraph $\Cay(R,S)$ is a \emph{Digraphical Regular Representation} (\emph{DRR}  for short) if $\Aut(\Cay(R,S))= \hat{R}$. If it is also a Cayley graph, then it is a \emph{Graphical Regular Representation} or \emph{GRR}.
\end{definition}

It is usually not easy to determine whether a Cayley digraph is a DRR, mostly because it is not easy to calculate the automorphism group. On the other hand, there is a particular subgroup of the automorphism group that is easier to understand. We first introduce some notation. Given a permutation group $G$ acting on a set $\Omega$ and $S \subseteq \Omega$,  we denote by $G_S$  the subgroup of $G$ that fixes $S$ setwise.  Given two subgroups $X$ and $Y$ of a common overgroup, we denote by $\N_Y(X)$ ($\Cent_Y(X)$) the normaliser (centraliser) of $X$ in $Y$.

\begin{theorem}\cite[Lemma 2.1]{Godsil1}\label{aut-R-S}
Let $R$ be a group,  let $S\subseteq R$ and let $A=\Aut(\Cay(R,S))$. Then $\N_A(\hat{R})=\hat{R}\rtimes \Aut(R)_S$.
\end{theorem}

Generally speaking, given a Cayley digraph $\Cay(R,S)$, calculating $\hat{R}\rtimes \Aut(R)_S$ is relatively easy, especially compared to determining $\Aut(\Cay(R,S))$. We are interested in groups for which knowing $\hat{R}\rtimes \Aut(R)_S$ is enough to decide whether $\Cay(R,S)$ is a DRR.



\begin{definition}
A group $R$ is \emph{DRR-detecting} if, for every subset $S$ of $R$, $\Aut(R)_S=1$ implies that $\Cay(R,S)$ is a DRR. It is \emph{GRR-detecting} if, for every inverse-closed subset $S$ of $R$, $\Aut(R)_S=1$ implies that $\Cay(R,S)$ is a GRR.
\end{definition}

Clearly, every DRR-detecting group is GRR-detecting. If $\Aut(R)_S=1$ but $\Cay(R,S)$ is not a DRR (respectively, not a GRR), then we say that $\Cay(R,S)$ \emph{witnesses that $R$ is not DRR-detecting} (respectively, \emph{not GRR-detecting}). Equivalently by Theorem~\ref{aut-R-S}, $\Cay(R,S)$ witnesses that $R$ is not DRR-detecting if $\hat{R}$ is self-normalising in $\Aut(\Cay(R,S))$ but $\Cay(R,S)$ is not a DRR.

We would like to determine which groups are DRR-detecting or GRR-detecting. Previous work on this topic includes a result by Godsil \cite{Godsil1} that if $p$ is prime, then every $p$-group that admits no homomorphism onto the wreath product $\C_p \wr \C_p$ is DRR-detecting. In particular, every abelian $p$-group is DRR-detecting. In~\cite{MMV-det}, with D.~Morris we showed that this result is sharp in the sense that $\C_p \wr \C_p$ is not GRR-detecting (or DRR-detecting) when $p$ is odd. We also proved that if a DRR-detecting group is nilpotent, then it is a $p$-group. In this paper we determine which groups of squarefree order are GRR-detecting, and which are DRR-detecting:

\begin{theorem}\label{theo:main}
Let $R$ be a group of squarefree order. 
\begin{enumerate}
\item If $|R|$ is prime, then $R$ is DRR-detecting (and therefore is GRR-detecting).
\item If $|R|$ has two prime factors, then:
\begin{enumerate} 
\item $R$ is not GRR-detecting (and therefore not DRR-detecting) if $R \cong \C_q \rtimes \C_r$ and either:
\begin{enumerate}
\item $(q,r)=(31,5)$; or
\item $(q,r)$ is a safe/Sophie Germain prime pair with $q \equiv 3 \pmod{4}$ and $q \ge 11$.
\end{enumerate}
\item $R$ is GRR-detecting but not DRR-detecting if:
\begin{enumerate}
\item $R$ is abelian; or
\item $R \cong \C_7 \rtimes \C_3$.
\end{enumerate}
\item $R$ is DRR-detecting (and therefore GRR-detecting) if $R$ does not fall into any of the above cases.
\end{enumerate}
\item If $|R|$ has at least three prime factors, then $R$ is not DRR-detecting, but is GRR-detecting if one of the following holds:
\begin{enumerate}
\item $R$ is abelian; 
\item $R \cong \D_{30}$; or
\item $R \cong \C_q \times \D_{2r}$ with $r \in \{3,5\}$.
\end{enumerate}
\end{enumerate}
\end{theorem}

(Throughout the paper, $\C_n$ denotes a cyclic group of order $n$ and $\D_{2n}$ a dihedral group of order $2n$.)

In Section~\ref{sec:WreathProducts}, we define generalised wreath products in the context of Cayley digraphs and show that they are never DRRs, which makes them very useful as potential witnesses that a group is not DRR-detecting (or GRR-detecting). We also give a sufficient condition to recognise a Cayley digraph as a generalised wreath product. Starting in  Section~\ref{sec:squarefreegen}, we restrict our attention to groups of squarefree order. We first show that ``most'' of these groups are not GRR-detecting and then deal with the remaining ``exceptional'' groups in Section~\ref{squarefreeExceptional}.

\section{Generalised wreath products}\label{sec:WreathProducts}
Our main approach to construct witnesses is to use generalised wreath products.  It is therefore important for us to understand what a generalised wreath product is in the context of Cayley digraphs.
\begin{definition}
Let $R$ be a group and let $S\subseteq R$. If there exist $K$ and $H$ with $1<K \trianglelefteq H<R$ such that 
\begin{equation}\label{doublecoset}\tag{$\star\star$}
K(S\setminus H)=S \setminus H=(S\setminus H)K
\end{equation}
then $\Cay(R,S)$ is a \emph{nontrivial generalised wreath product (with respect to $K$ and $H$)}. If $H=K$, then $\Cay(R,S)$ is a \emph{nontrivial wreath product (with respect to $K$)}.
\end{definition}
It is not hard to see that, if $K \trianglelefteq R$ or $S=S^{-1}$, then (\ref{doublecoset}) is equivalent to $K(S\setminus H)=S \setminus H$. We will often use this fact throughout the paper.


\begin{lemma}\label{GenWreath}
A Cayley digraph that is a nontrivial generalised wreath product is not a DRR.
\end{lemma}
\begin{proof}
Let $\Gamma=\Cay(R,S)$ be a nontrivial generalised wreath product with respect to $K$ and $H$. By definition, we have $1<K \trianglelefteq H<R$ and $K(S\setminus H)=S \setminus H=(S\setminus H)K$.

Let $k\in K$ and let $\alpha_k\in\Sym(R)$ be the map which right multiplies elements of $H$ by $k$ while fixing all other elements of $R$. We claim that $\alpha_k\in\Aut(\Gamma)$. Let $(y,x)$ be an arc of $\Gamma$, so that $xy^{-1}\in S$. We check that $\alpha_k(x)\alpha_k(y)^{-1}\in S$. If $x,y\notin H$, this is trivial. Similarly, if $x,y\in H$, then $\alpha_k(x)\alpha_k(y)^{-1}=xk(yk)^{-1}=xy^{-1}\in S$. Now, suppose that $x\in H$ and $y\notin H$.  In particular, $xy^{-1}\notin H$. We have 
$$\alpha_k(x)\alpha_k(y)^{-1}=xky^{-1}=k^{x^{-1}}xy^{-1}\in Kxy^{-1}\subseteq K(S\setminus H)=S\setminus H,$$
 as required (we used the fact that $K \trianglelefteq H$ and $x \in H$). Finally, if $x\notin H$ and $y\in H$,  then again $xy^{-1}\notin H$ and 
$$\alpha_k(x)\alpha_k(y)^{-1}=x(yk)^{-1}=xk^{-1}y^{-1}=xy^{-1}(k^{-1})^{y^{-1}}\in xy^{-1}K\subseteq (S\setminus H)K=S\setminus H.$$

Since $1<K$, there is some $k$ such that $\alpha_k\neq 1$ and, since $H<R$, it follows that $\alpha_k\notin\hat{R}$ and $\Gamma$ is not a DRR.
\end{proof}

We have the following immediate corollary.

\begin{corollary}\label{DRRProp1}
Let $R$ be a group, let $S\subseteq R$ and let $K$ and $H$ be such that
\begin{enumerate}
\item  $1< K\unlhd H< R$,
\item $K(S\setminus H)=S\setminus H=(S\setminus H)K$, and
\item $\Aut(R)_S=1$.
\end{enumerate}
Then $\Cay(R,S)$ witnesses that $R$ is not DRR-detecting (and not GRR-detecting if $S=S^{-1}$).
\end{corollary}

To apply Corollary~\ref{DRRProp1}, one must show that $\Aut(R)_S=1$. To do this, it will often be  easiest to show first that $\Aut(R)_S$ normalises $H$, so that $\Aut(R)_S=\Aut(R)_{S \cap H} \cap \Aut(R)_{S \setminus H}$. The most obvious situation in which $\Aut(R)_S$  normalises $H$ is if $\Aut(R)$ itself normalises $H$; that is, when $H$ is characteristic in $R$.  Here is another approach that can also be used.

\begin{proposition}\label{char-H-blocks}
Let $R$ be a group, let $S\subseteq R$ and let $1<K <H <R$. If
\begin{enumerate}
\item $K$ characteristic in $R$,
\item $K$ is maximal in $H$,
\item $K(S\setminus H)=S \setminus H$, and \label{hyp3}
\item it is not the case that $K(S\setminus K)=S \setminus K$,  \label{hyp4}
\end{enumerate} 
then $\Aut(R)_S$ normalises $H$.
\end{proposition}
\begin{proof}
Let $H'=\langle s \in S: Ks \not\subseteq S\rangle$. Since $K$ is characteristic in $R$, $\Aut(R)$ normalises $K$, and therefore so does $\Aut(R)_S$.  It follows that $\Aut(R)_S$ normalises $H'$. By (\ref{hyp3}) we have $H' \le H$, and by (\ref{hyp4}), $H'\nleq K$. It follows that $K< KH'\leq H$. Since $K$ is maximal in $H$, $H=KH'$ hence $\Aut(R)_S$ normalises $H$, as required.
\end{proof}

We end this section with a sufficient condition to recognise a Cayley digraph as a nontrivial generalised wreath product. (For $G\leq \Aut(\Cay(R,S))$, we denote by $G_1$ the stabiliser of the vertex of $\Cay(R,S)$ corresponding to the identity of $R$.)



\begin{lemma}\label{GenWreathCosets}
Let $R$ be a group, let $S\subseteq R$, let $G\leq \Aut(\Cay(R,S))$ and let $K$ and $H$ be such that $1<K \trianglelefteq H<R$ and $H\leq \N_R(G_1)$. If, for every $r\in R\setminus H$, we have $G_1K,G_1K^r\subseteq G_1G_1^r$, then $\Cay(R,S)$ is a nontrivial generalised wreath product with respect to $K$ and $H$. Moreover, if $S=S^{-1}$, then $G_1K \subseteq G_1G_1^r$ is sufficient to reach the same conclusion.
\end{lemma}

\begin{proof}
Since $H\leq \N_R(G_1)$, we have $G_1H^{G_1}=G_1HG_1=G_1H$. Note that $G\leq \Aut(\Cay(R,S))$ implies that $S^{G_1}=S$  hence $G_1(S \setminus H)^{G_1}=G_1 (S \setminus H)$. If, for every $r \in R \setminus H$, we have $G_1K \subseteq G_1G_1^r$, then we have $$G_1Kr^{-1} \subseteq G_1 G_1^r r^{-1}=G_1r^{-1}G_1=G_1(r^{-1})^{G_1}$$
and it follows that $G_1K( S \setminus H) \subseteq G_1 (S \setminus H)^{G_1}=G_1(S \setminus H)$, which implies $K(S \setminus H) =S \setminus H$. Likewise, if for every $r \in R \setminus H$, we have $G_1K^r \subseteq G_1G_1^r$, then $$G_1r^{-1}K \subseteq G_1r^{-1}G_1=G_1(r^{-1})^{G_1}$$
and it follows that $G_1( S \setminus H)K \subseteq G_1 (S \setminus H)^{G_1}=G_1 (S \setminus H)$, which implies $(S \setminus H)K =S \setminus H$.

Hence, if $G_1K,G_1K^r\subseteq G_1G_1^r$ for every $r\in R\setminus H$, then $K(S\setminus H)=S \setminus H=(S\setminus H)K$ hence $\Cay(R,S)$ is a nontrivial generalised wreath product. 
Moreover, if $S=S^{-1}$, then as previously noted, $K(S\setminus H)=S \setminus H$ suffices to reach the same conclusion.  
\end{proof}

\section{Groups of squarefree order, generic case}\label{sec:squarefreegen}

The structure of  groups of squarefree order has been well understood since the work of H\"older~\cite{HolderSquarefree}.  An obvious observation is that every subgroup of such a group is a Hall subgroup hence every normal subgroup is characteristic. H\"older proved that these groups are metacyclic. In particular, if $R$ is a group of squarefree order, we have $R\cong \C_t \times (\C_n\rtimes \C_m)$, where $\Z(\C_n\rtimes \C_m)=1$ (and $t$, $n$ and $m$ are pairwise coprime). (We use the usual notation $\Z(G)$ for the centre of the group $G$.) An easy consequence of this is the fact that, for every set of primes $\pi$ dividing $|R|$, $R$ has a Hall $\pi$-subgroup. Finally, we will make frequent use of the fact that if $p$ and $q$ are primes with $q>p$ and $X$ is a nonabelian subgroup of order $pq$, then  $q\equiv 1 \pmod p$. We will also make use of the following result.

\begin{lemma}\label{some-sources-cover}
Let $R\cong \C_n\rtimes \C_m$ be a group of squarefree order with trivial center. Suppose that, for every pair of primes $p$ and $q$ with $p \mid m$ and $q \mid n$, every subgroup of $R$ of order $pq$ is nonabelian. Let $H$ be a characteristic subgroup of prime index in $R$. If $m$ is not prime, then $\Cent_{\Aut(R)}(H)=1$.
\end{lemma}
\begin{proof}
Let $x$ and $y$ be elements of order $m$ and $n$ in $R$, respectively and let $Y=\langle y\rangle$. Note that $Y$ is a characteristic subgroup of $R$ and there exists an integer $j$ such that, for every $y'\in Y$, we have $(y')^x=(y')^{j}$. Now, if $\alpha\in\Aut(R)$, then $x^\alpha$ must also have this property, that is, $(y')^{x^\alpha}=(y')^{j}$ for every $y'\in Y$. An easy calculation shows that this implies that $x^\alpha\in Yx$, say $x^\alpha=y^ix$. Let $p=|R:H|$. Note that, since $H$ is normal in $R$, we have $Y\leq H$ and $p$ divides $m$.  Let $\alpha\in \Cent_{\Aut(R)}(H)$. Since $m$ is not prime, we have $x^p\in H\setminus Y$ so $(x^p)^\alpha=x^p$. Since $x^p$ commutes with $x$, $(x^p)^\alpha=x^p$ commutes with $x^\alpha=y^ix$. It follows that $x^p$ commutes with $y^i$.  If $y^i\neq 1$, then by hypothesis, $\Cent_R(y^i)=Y$ but $x^p\notin Y$ hence we must have $y^i=1$. This implies that $x^\alpha=x$  and $\alpha=1$, as required.
\end{proof}

In our first main result, Theorem~\ref{main-GRR}, we construct Cayley graphs $\Cay(R,S)$ that are nontrivial generalised wreath products with respect to some subgroups $K$ and $H$ of $R$, and that witness that $R$ is not GRR-detecting. One component of our construction involves taking a GRR on $H$ when possible. In order to understand when this is possible, we need to know which groups of squarefree order admit GRRs. Although many researchers including Watkins, Imrich, Nowitz, and Hetzel made significant contributions along the way, the ultimate result about which groups admit GRRs is due to Godsil. We provide a statement of his result that makes it easy to see which of these groups have squarefree order.

\begin{theorem}\cite{Godsil}\label{Godsil-GRRs}
Every group admits a GRR except:
\begin{itemize}
\item abelian groups of exponent greater than $2$;
\item generalised dicyclic groups (which have orders divisible by $4$);
\item the dihedral groups $\D_6$ and $\D_{10}$; and
\item eleven other small groups, none of whose orders is squarefree.
\end{itemize}
\end{theorem}
It follows that the only nonabelian groups of squarefree order that do not admit a GRR are $\D_6$ and $\D_{10}$. We are now ready to show that ``most'' groups of squarefree order are not GRR-detecting (and thus not DRR-detecting). 
(We do set aside a number of special cases for further consideration in Section~\ref{squarefreeExceptional}.)

\begin{theorem}\label{main-GRR}
Let $R$ be a group of squarefree order. If $R$ is not abelian and $R\notin \{\D_6, \D_{10},\D_{30}, \D_6 \times \C_q, \D_{10}\times \C_q, \C_q\rtimes \C_p: p,q \text{ primes}\}$, then $R$ is not GRR-detecting. 
\end{theorem}
\begin{proof}
We can assume that $|R|$ has at least three prime divisors, since we have excluded the other possibilities. Let $R\cong \C_t \times (\C_n \rtimes \C_m)$, where $\Z(\C_n \rtimes \C_m)=1$. Since $R$ is nonabelian, we have $n,m\geq 2$.  We now split the proof into two cases.

\medskip

{\noindent\textbf{Case 1: For every pair of primes $p$ and $q$ with $p \mid m$ and $q \mid nt$, every subgroup of $R$ of order $pq$ is nonabelian.}}

In this case, we have $t=1$. Let $p$ be the smallest prime dividing $m$. If $m\neq p$, then take $H$ to be the characteristic subgroup of index $p$ in $R$, so $H\cong \C_n\rtimes \C_{m/p}$, with $m/p\geq 3$. In particular, $H$ admits a GRR. Let $S$ be the connection set for a GRR on $H$.  Since $H$ is characteristic in $R$, it is fixed by $\Aut(R)_S$ and thus so is $S\cap H=S$. Since $\Cay(H, S \cap H)$ is a GRR, it follows that $\Aut(R)_S\leq \Cent_{\Aut(R)}(H)$. By Lemma~\ref{some-sources-cover}, $\Cent_{\Aut(R)}(H)=1$ hence $\Aut(R)_S=1$ and by Corollary~\ref{DRRProp1} (applied with $K=H$), $R$ is not GRR-detecting.

We may thus assume that $m=p$. Since $|R|$ has at least three prime divisors and $t=1$, $n$ is not prime. Let $q$ be the largest prime dividing $n$ and write $n=qn'$. Recall that every prime divisor of $n$ must be $1$ modulo $p$. If $p\geq 3$, then it immediately follows that $q\geq 7$. If $p=2$, the only other possibility is $(q,n)=(5,15)$, but this is excluded by our hypothesis, hence $q\geq 7$ in either case. Let $K$ be the characteristic subgroup of order $q$ in $R$, and $H$ a subgroup of order $pq$. Note that $H$ is nonabelian and, since $q\geq 7$, $H$ admits a GRR.

Let $k \in K$ have order $q$, let $h \in H$ have order $p$, and let $x \in R\setminus H$ have order $n'$. Let $S'$ be the connection set for a GRR on $H$, and let $S=S' \cup Khx \cup K(hx)^{-1}$. Suppose that $K(S\setminus K)=S\setminus K$. This implies that $K(S'\setminus K)=S'\setminus K$ and then Lemma~\ref{GenWreath} implies that $\Cay(H,S')$ is not a DRR, a contradiction. It follows that $K(S\setminus K)\neq S\setminus K$. Let $\alpha \in\Aut(R)_S$. By Proposition~\ref{char-H-blocks}, $\alpha$ normalises $H$, so since $\Cay(H, S \cap H)$ is a GRR, $\alpha$ fixes $H$ pointwise. The neighbourhood of $h \in H$ outside $H$ is $Khxh\cup Kx^{-1}=Kx^{h^{-1}}h^2\cup Kx^{-1}$, which must therefore be fixed setwise by $\alpha$. Note that $Kx^{-1}$ has a unique element of order $n'$, namely $x^{-1}$, and all its other elements have order $n$. If $p=2$, then $Khxh=Kx^{-1}$, whereas if $p\geq 3$, then every element of $Kx^{h^{-1}}h^2$ has order $p$. Either way,  $\alpha$ must fix $x^{-1}$ and thus fix $R$ pointwise so $\alpha=1$. We conclude that $\Aut(R)_S=1$ and, by Corollary~\ref{DRRProp1}, $R$ is not GRR-detecting. 

\medskip

{\noindent\textbf{Case 2: There exists primes $p$ and $q$ with $p \mid m$ and $q \mid nt$ such that $R$ has a cyclic subgroup of order $pq$.}}

Choose $p$ and $q$ satisfying the above with $p$ as large as possible. Since $R$ is nonabelian,  $nt$ is not prime. Let $H$ be the characteristic subgroup of index $p$ in $R$. Since $nt$ is not prime, $H$ is not isomorphic to $\D_6$ or $\D_{10}$, hence either $H$ admits a GRR or $H\cong \C_{nt}$.

Let $r=nt/q$. Note that every  element of order $p$ in $R$ must act nontrivially on some subgroup of the cyclic subgroup of order $r$ in $R$. We show that $r>5$. Suppose, by contradiction, that $r\leq 5$. This implies $p=2$ and $r \in \{3,5\}$. If $m=p=2$, then $R \cong \D_{2r} \times \C_q$, a case we have excluded by hypothesis. So $m>p$ and there is a prime $p'$ dividing $m$ with $p'>2$. Since $r\in \{3,5\}$, $R$ has a cyclic subgroup of order $p'r$ but this contradicts our choice of $p$. It follows that $r>5$.

Let $k$, $g$ and $x$ be elements of order $q$, $r$ and $p$ in $R$, respectively. Let $K=\langle k\rangle$ and $B=\langle kg\rangle$. Note that $B\cong \C_{nt}$ and that $K$ and $B$ are both characteristic in $R$ and contained in $H$. If $H$ is nonabelian, then take $S'$ to be the connection set for a GRR on $H$; if $H \cong \C_{nt}$ then take $S'=\{kg,(kg)^{-1}\}$.  Let $S=S' \cup Kx^{\pm1} \cup K(gx)^{\pm1} \cup K(g^3x)^{\pm1}$. Note that these really are three distinct cosets of $K$ since $|g|=r>5$. Note also that $S\cap H=S'$ and that $S\setminus S'\subseteq Bx^{\pm1}$. 

Let $\alpha \in \Aut(R)_S$. Since $H$ is characteristic in $R$,  $H^\alpha=H$ hence $(S')^\alpha=S'$. By our choice of $S'$, it follows that $(kg)^\alpha=(kg)^{\pm1}$ which implies $g^\alpha=g^{\pm1}$. Write $x^\alpha=bx^{\epsilon}$, with $b\in B$ and $\epsilon=\pm 1$. Note that 
$$(g^\alpha)^x=(g^x)^\alpha=(g^\alpha)^{x^\alpha}=(g^\alpha)^{bx^{\epsilon}}=(g^\alpha)^{x^{\epsilon}}.$$
This implies that $\epsilon=1$, so $x^\alpha\in Bx$ and therefore $\alpha$ fixes $Kx \cup Kgx \cup Kg^3x$. Since $K^\alpha=K$, $\alpha$ must permute these three $K$-cosets. Write $(Kx)^\alpha=Kg^ix$ with $i \in \{0,1,3\}$.  It follows that $(Kgx)^\alpha = (gKx)^\alpha = g^{\pm 1}Kg^i x=Kg^{i\pm1}x$ and $i\neq 3$. Moreover, if $i=1$ then $g^\alpha=g^{-1}$ and $\alpha$ interchanges $Kg$ and $Kgx$, so must fix $Kg^3x$, so $Kg^3x=(Kg^3x)^\alpha=Kg^{-3+1}x$ and $g^5=1$, contradicting $|g|=r>5$. It follows that $i=0$ and $\alpha$ fixes $Kx$, $Kgx$, and $Kg^3x$. Since $x$ and $k$ commute, $x$ is the unique element of order $p$ in $Kx$, so it is fixed by $\alpha$. Similarly, $gx$ is the unique element of $Kgx$ whose order is not a multiple of $q$, so it too is fixed by $\alpha$ hence so is $g$. It follows that $\alpha$ centralises $H$ and $\alpha=1$. This shows that $\Aut(R)_S=1$ and it follows from Corollary~\ref{DRRProp1} that  $R$ is not GRR-detecting.

\end{proof}


\section{Groups of squarefree order, exceptional cases}\label{squarefreeExceptional}

In this section we proceed through the groups that were excluded in the hypothesis of Theorem~\ref{main-GRR}.  We begin with the three sporadic groups. The following result can be checked by computer.

\begin{proposition}\label{prop:sporadic}
$\D_6$ and $\D_{10}$ are DRR-detecting (and therefore GRR-detecting). $\D_{30}$ is GRR-detecting but not DRR-detecting.
\end{proposition}

We next deal with abelian groups. We divide these into two classes, according to whether or not their order is prime. 

\begin{proposition}\label{prop:primeorder}
Groups of prime order are DRR-detecting (and therefore GRR-detecting).
\end{proposition}
\begin{proof}
Let $R$ be a group of prime order and let $S\subseteq R$. It is known that either $\hat{R}$ is normal in $\Aut(\Cay(R,S))$ or $\Aut(\Cay(R,S))$ is doubly transitive (see for example~\cite[Theorem 11.7]{Wielandt}). In the latter case, $\Cay(R,S)$ is a complete graph and $\Aut(\Cay(R,S))=\Sym(R)$. In either case, $\Aut(R)_S=1$ implies that $\Aut(\Cay(R,S))=\hat{R}$, and $R$ is DRR-detecting.
\end{proof}

\begin{proposition}\label{prop:abelian}
Abelian groups of squarefree composite order are GRR-detecting.
\end{proposition}
\begin{proof}
Let $R$ be an abelian group of squarefree composite order and let $S\subseteq R$ with $S=S^{-1}$. Inversion is a non-identity automorphism of $R$ that preserves $S$ hence $\Aut(R)_S\neq 1$. This shows that $R$ is GRR-detecting.
\end{proof}

We still need to show that abelian groups of squarefree composite order are not DRR-detecting. In order to do so, we will use the following two results.

\begin{theorem}\cite[Theorem~1.9]{MMV-det}\label{direct-prod}
If $G_1$ and $G_2$ are nontrivial groups that admit a DRR (a GRR, respectively) and $\gcd(|G_1|,|G_2|)=1$, then $G_1 \times G_2$ is not DRR-detecting (not GRR-detecting, respectively).
\end{theorem}

To apply Theorem~\ref{direct-prod}, we need to understand which groups of squarefree order admit DRRs. We use the following result of Babai.

\begin{theorem}\cite[Theorem 2.1]{Babai}\label{DRRs}
Every group admits a DRR except $\C_2^i$ for $2 \le i \le 4$, $\C_3^2$, and $\Q_8$. In particular, every group of squarefree order admits a DRR.
\end{theorem}

\begin{corollary}\label{cor:exceptional}
Let $R$ be a group of squarefree order. If $R$ is abelian of composite order or $R\cong \D_{2r}\times \C_q$ with $r\in\{3,5\}$ and $q$ prime, then $R$ is not DRR-detecting.
\end{corollary}
\begin{proof}
Note that $R$ is a nontrivial direct product of two groups of coprime squarefree order. The result then follows from Theorems~\ref{direct-prod} and~\ref{DRRs}. 
\end{proof}


To prove Theorem~\ref{theo:main}, it remains to show that $\C_q \times \D_6$ and $\C_q \times \D_{10}$ are GRR-detecting and to determine the status of nonabelian groups whose order is a product of two primes. This is our goal in the remainder of the paper.

\subsection{The case when $\Cay(R,S)$ is a generalised wreath product}

Since we are trying to show that some groups are DRR or GRR-detecting, we have to show that they do not admit witnesses. One case that needs to be handled is to show that even nontrivial generalised wreath products on these groups are not witnesses. This is the goal of this subsection.

\begin{lemma}\label{UsingConjugation}
If  $\Cay(R,S)$ is a nontrivial generalised wreath product with respect to $K$ and $H$ and $\Z(H)\cap K\nleq \Z(R)$, then $\Aut(R)_S > 1$.
\end{lemma}
\begin{proof}
By definition, $K(S\setminus H)=S \setminus H=(S\setminus H)K$. Let $k$ be an element of $(\Z(H)\cap K)\setminus \Z(R)$ and let $\alpha_k\in\Aut(R)$ denote conjugation by $k$. Since $k\in \Z(H)$, we have that $\alpha_k$ fixes $H$ pointwise. Moreover, since $k\in K$, for every $s\in S\setminus H$, we have $s^{\alpha_k}=k^{-1}sk\in KsK\subseteq S$. It follows that $\alpha_k\in \Aut(R)_S$. Finally, as $k\notin \Z(R)$, $\alpha_k\neq 1$, as required.
\end{proof}

From this we are able immediately to prove our desired result in the case where $|R|$ is a product of two primes.

\begin{corollary}\label{pq-wreath}
Let $R$ be a nonabelian group whose order is a product of two distinct primes and let $S\subseteq R$. If $\Cay(R,S)$ is a nontrivial generalised wreath product, then $\Aut(R)_S > 1$.
\end{corollary}
\begin{proof}
Say that $\Cay(R,S)$ is a nontrivial generalised wreath product with respect to $K$ and $H$, so that  $1<K \trianglelefteq H<R$.  Given the order of $R$, we must have $H=K$ of prime order, with $\Z(H)=H\nleq \Z(R)=1$, and the result follows by Lemma~\ref{UsingConjugation}.
\end{proof}

It remains to deal with groups of the form $\C_q \times \D_{2r}$ where $r \in \{3,5\}$. We first need the following result, which is easy but we include a proof for completeness.

\begin{lemma}\label{dihedral-auts}
Let $r \in \{3,5\}$, let $D=\D_{2r}$ and let $S \subseteq D$ with $S=S^{-1}$. Then there exists some nontrivial $\beta \in \Aut(D)_S$ that inverts every element of the subgroup of order $r$ of $D$.
\end{lemma}
\begin{proof}
Let $C$ be the (cyclic) subgroup of order $r$ of $D$ and let $x\in D\setminus C$. We show that there exists an element $z\in xC$ such that conjugation by $z$ fixes $S$ setwise. The result then follows.

Clearly, conjugation by an element of $xC$ inverts every element of $C$ (and hence preserves $S\cap C$), so it suffices to show that $S\cap (D\setminus C)$ is preserved by conjugation by an element of $xC$. This is equivalent to preserving the complement $(D\setminus C)\setminus S$. Since $|D\setminus C|\leq 5$, we assume without loss of generality that $|S\cap (D\setminus C)|\leq 2$.

If $|S\cap (D\setminus C)|=0$, there is nothing to prove. If $|S\cap (D\setminus C)|=1$, then just take $z\in S\cap (D\setminus C)$. Finally, assume that $|S\cap (D\setminus C)|=2$, say $S\cap (D\setminus C)=\{ xy^i,xy^j\}$ where $C=\langle y\rangle$. Let $z=xy^{(i+j)/2}$ (where $(i+j)/2$ is computed in $\ZZ_r$.) One can check that conjugation by $z$ interchanges $xy^i$ and $xy^j$ hence preserves $S$, as required.
\end{proof}

\begin{proposition}\label{qDrWreath}
Let $r \in \{3,5\}$, let $q$ be an odd prime distinct from $r$,  let $R=\C_q \times \D_{2r}$ and let $S\subseteq R$ with $S=S^{-1}$. If $\Cay(R,S)$ is a nontrivial generalised wreath product with respect to $K$ and $H$, and $K$ has prime order, then $\Aut(R)_S > 1$. 
\end{proposition}
\begin{proof}
Write $R=\langle z \rangle \times (\langle y\rangle \rtimes \langle x \rangle)$, with $|z|=q$, $|y|=r$ and $|x|=2$. Note that $\Z(R)=\langle z\rangle$. Up to conjugacy, we can assume that $K$ is generated by one of $x$, $y$ or $z$. As for $H$, we can assume that it is maximal in $R$ with respect to being proper in $R$ and normalising $K$. Indeed, if $\Cay(R,S)$ is a nontrivial generalised wreath product with respect to $K$ and $H'$, with $H'\leq H$, then $S\setminus H'=K(S\setminus H')=(S\setminus H')K$ so $S\setminus H=K(S\setminus H)=(S\setminus H)K$ and $\Cay(R,S)$ is also a nontrivial generalised wreath product with respect to $K$ and $H$. We thus only have to consider the following cases.
\begin{enumerate}
\item $H=\langle x,y\rangle\cong \D_{2r}$ and $K=\langle y\rangle\cong \C_r$. By Lemma~\ref{dihedral-auts}, there exists some nontrivial $\beta\in\Aut(H)_{S\cap H}$ that acts by inversion on $K$. In particular $K^\beta=K$. Let $\alpha$ be the unique automorphism of $R$ that fixes $z$ and agrees with $\beta$ on $H$. Note that $\alpha$ preserves both $K$ and $H$ so $\alpha$ preserves the two $K$-cosets in $H$. If $s \in S\cap H$ then $s^\alpha=s^\beta \in S\cap H$. If $s \in S \setminus H$, say $s\in z^ihK$, for some $h\in H$, then  $s^\alpha\in z^i(hK)^\alpha=z^ihK\subseteq S\setminus H$, so $\alpha \in \Aut(R)_S$, as required.

\item $H=\langle y,z\rangle\cong \C_{qr}$ and $K=\langle y\rangle\cong \C_r$. In this case,  $\Z(H)\cap K=K\nleq \Z(R)$ and the result follows by Lemma~\ref{UsingConjugation}.

\item $H=\langle y,z\rangle\cong \C_{qr}$ and $K=\langle z\rangle\cong \C_q$. 
Let $\pi:R\to R/K$ be the canonical projection mapping. Note that $\pi(R)\cong \D_{2r}$ and $\pi(S)=\pi(S)^{-1}$ so, by Lemma~\ref{dihedral-auts},  there exists some nontrivial $\beta\in\Aut(\pi(R))_{\pi(S)}$ that inverts every element of $\pi(\langle y\rangle)$. Let $\alpha$ be the unique automorphism of $R$ that inverts $z$ and such that $\pi\beta=\alpha\pi$. (In other words, $s^\alpha K=(sK)^\beta$ for every $s\in R$.) Since $H$ is characteristic in $R$, we have $H^\alpha=H$. As $\beta$ inverts $\pi(\langle y\rangle)$, we have $y^\alpha K=(yK)^\beta=y^{-1}K$. Moreover,  $\langle y\rangle$ is characteristic in $R$ hence $y^\alpha\in \langle y\rangle\cap y^{-1}K$ and $y^\alpha=y^{-1}$. It follows that $\alpha$ acts by inversion on $H$. Since $S=S^{-1}$, $\alpha$ preserves $S\cap H$. Now, if $s\in S\setminus H$, then we have $sK\in \pi(S)$ and since $\beta$ preserves $\pi(S)$, it follows that $s^\alpha K=(sK)^\beta \in \pi(S)$. As $K(S\setminus H)=S\setminus H$, we have that $s^\alpha K\in S\setminus H$. This shows that $\alpha\in\Aut(R)_S$.

\item  $H=\langle x,z\rangle\cong \C_{2q}$ and $K=\langle z\rangle\cong \C_q$. Let $\alpha: R\to R$ be defined by $(x^iy^jz^k)^\alpha=x^iy^{j}z^{-k}$. Note that $\alpha\in\Aut(R)$. Moreover, $\alpha$ acts by inversion on $H$. Since $S=S^{-1}$, $S \cap H$ is fixed by $\alpha$. If $s \in S \setminus H$, say $s=x^iy^jz^k$, then  $s^\alpha=x^iy^jz^{-k}\in sK\subseteq S\setminus H$ hence $S \setminus H$ is also fixed by $\alpha$ and $\alpha\in\Aut(R)_S$, as required.

\item $H=\langle x,z\rangle\cong \C_{2q}$ and $K=\langle x\rangle\cong \C_2$. In this case,  $\Z(H)\cap K=K\nleq \Z(R)$ and the result follows by Lemma~\ref{UsingConjugation}.
\end{enumerate}
\end{proof}

\subsection{Main results}

We are at last ready to show that groups of the form $\C_q \times \D_{2r}$ are GRR-detecting  and to characterise DRR-detection and GRR-detection for nonabelian groups whose order is a product of two primes. We first prove the following well known result:

\begin{lemma}\label{OpAffinePrimitive}
If $G$ is a primitive group of affine type with socle an elementary abelian $p$-group, then a point-stabiliser has no non-trivial normal $p$-subgroup.
\end{lemma}
\begin{proof}
Let $V$ be the socle of $G$ and, to arrive at a contradiction let $T$ be a non-trivial normal $p$-subgroup of $G_x$. Since $T$ is normal in $G_x$, $\Cent_G(T)$ is normalised by $G_x$. It follows that $Z=\Cent_V(T)$ is also normalised by $G_x$. Now,  $V$ and $T$ are both $p$-groups, so $1<Z$ but since $V$ is a regular subgroup of the permutation group $G$, $\Cent_{G_x}(V)=1\neq T$ hence $Z<V$. It follows that the orbits of $Z$ form a non-trivial system of imprimitivity for $G$, contradicting its primitivity.
\end{proof}

The rest of the proof is split into two: Theorem~\ref{theo:SpecialCases} which essentially reduces the problem to the almost simple case, and Corollary~\ref{cor:special} which handles that case.

\begin{theorem}\label{theo:SpecialCases}
Let $q,r$ be distinct primes and either 
\begin{itemize}
\item $r\in \{3,5\}$, $q$ is odd and let $R\cong \C_q\times \D_{2r}$, or
\item   $R$ is a nonabelian group isomorphic to $\C_q\rtimes \C_r$.
\end{itemize}
Let $S\subseteq R$, suppose that $S=S^{-1}$ when $R\cong \C_q\times \D_{2r}$ and let $G$ satisfy $\hat{R}<G\leq\Aut(\Cay(R,S))$. If $\hat{R}$ is maximal in $G$, then one of the following occurs:
\begin{enumerate}
\item $\Aut(R)_S>1$,
\item $\hat{R}$ is core-free in $G$ and $G$ is almost simple, or 
\item $R\cong \C_q\times \D_{2r}$, $\C_q$ is the core of $\hat{R}$ in $G$, and $G/\C_q$ is almost simple.
\end{enumerate}
\end{theorem}

\begin{proof}
For $X\leq \Aut(\Cay(R,S))$, let $X_1$ denote the stabiliser in $X$ of the vertex of $\Cay(R,S)$ corresponding to the identity of $R$. For simplicity, we will identify $\hat{R}$ with $R$ from now on. Note that $G_1$ is non-trivial and core-free in $G$ and $G=RG_1$ with $R\cap G_1=1$. Let $N$ be the core of $R$ in $G$. If $R$ is normal in $G$, then $1<G_1\leq \Aut(R)_S$, hence we  assume that $R$  is not normal in $G$ and $N<R$. Let $\overline{G}=G/N$, $\overline{R}=R/N$ and $\overline{G_1}=G_1N/N\cong G_1$.  Note that $\overline{R}$ is a maximal core-free subgroup of $\overline{G}$, so we can view $\overline{G}$ as a primitive group with point-stabiliser $\overline{R}$ and a regular subgroup $\overline{G_1}$. Since the point-stabiliser $\overline{R}$ is soluble, the primitive type of $\overline{G}$ is either  affine, almost simple, or product action. Moreover, because the order of the point-stabiliser $\overline{R}$ is squarefree, the product action case can't occur. (See for example \cite[Theorem 1.1]{LiZhang} for both of these claims.)

Suppose first that $\overline{G}$ is almost simple. In this case, the point-stabiliser $\overline{R}$ cannot be abelian (see for example \cite[Lemma 2.1]{CayleyAbelian}), so either $N=1$ (and $G\cong\overline{G}$, $R\cong\overline{R}$, and conclusion (2) holds, completing the proof) or $N\cong \C_q$ and $\overline{R}\in\{\D_6,\D_{10}\}$. In the latter case, conclusion (3) holds, again completing the proof. 

From now on, we assume that $\overline{G}$ is of primitive affine type. In this case, there exists a normal subgroup $E$ of $G$ such that $N\leq E$, $\overline{G}=\overline{E}\rtimes \overline{R}$ and $\overline{E}\cong \C_p^x$ for some prime $p$. It follows that $|G_1|=|\overline{G_1}|=p^x$. Note that $\overline{R}$ acts faithfully and irreducibly on $\overline{E}$. We now prove the following claim.

\medskip

{\noindent\textbf{Claim:} If $p$ divides $|N|$ and  $R$ has a normal Sylow $p$-subgroup, then $\Aut(R)_S>1$.}

\medskip
Let $H=\N_R(G_1)$, let $K$ be a Sylow $p$-subgroup of $R$ and let $X$ be a Sylow $p$-subgroup of $E$. Since $p$ divides $|N|$, it does not divide $|\overline{R}|$, hence $\overline{E}$ is a normal Sylow $p$-subgroup of $\overline{G}$. Note that $K$ is characteristic in $N$ thus normal in $G$. It follows that $K\leq X$, $X/K=\overline{E}$ and $X$ is normal in $G$. Moreover, $|X:G_1|=p$, hence $G_1$ is normal and maximal in $X$ hence $K\leq H$. Let $r\in R\setminus H$. By definition, we have $G_1^r\neq G_1$ but $G_1$ is contained in $X$ which is normal in $G$, so $G_1^r\leq X$. Since $G_1$ is a normal maximum subgroup of $X$, $G_1G_1^r=X$. It follows that $G_1K,G_1K^r\subseteq X=G_1G_1^r$, and we can apply Lemma~\ref{GenWreathCosets} to conclude that $\Cay(R,S)$ is a nontrivial generalised wreath product with respect to $K$ and $H$. If $R \cong \C_q\rtimes \C_r$ then the claim follows by Corollary~\ref{pq-wreath}, whereas if $R \cong \C_q \times \D_{2r}$, it follows by Proposition~\ref{qDrWreath}.
\qedsymbol

\medskip

 We split the remainder of the proof into two cases, according to whether $N$ is  cyclic or dihedral.

\medskip

{\noindent\textbf{$N$ is cyclic:}}
Suppose that $p$ divides $|N|$. Since $N$ is cyclic, its Sylow $p$-subgroup is characteristic, therefore normal in $R$. We can thus apply our claim to conclude that $\Aut(R)_S>1$, completing the proof. From now on, we assume that $p$ does not divide $|N|$. Let $C$ be the centraliser of $N$ in $G$. Since $N$ is cyclic, we have $N\leq C$. If $N=C$, then $G/N$ embeds in $\Aut(N)$ which is abelian, a contradiction since $G/N$ is nonabelian. We conclude that $N<C$. Since $\overline{E}$ is the unique minimal normal subgroup of $\overline{G}$, we have $E\leq C$.

If $p$ is coprime to $|\overline{R}|$, then $\overline{E}$ is the unique Sylow $p$-subgroup of $\overline{G}$ so $\overline{G_1}=\overline{E}$ and $G_1\leq E\leq C$ which implies $E=NG_1=N\times G_1$. Since $p$ does not divide $|N|$,  $G_1$ is characteristic in $E$, and thus normal in $G$, a contradiction. This shows that $p$ divides $|\overline{R}|$. Recall that $\overline{R}$ has no non-trivial normal $p$-subgroup (Lemma~\ref{OpAffinePrimitive}), so we get the following cases:

\begin{enumerate}
\item $R\cong \C_q\rtimes \C_r$, $N=1$, $\overline{R}\cong \C_q\rtimes \C_r$ and $p=r$.\label{case1}
\item $R\cong \C_q\times \D_{2r}$, $N=1$, $\overline{R}\cong \C_q\times \D_{2r}$ and $p=2$. \label{case2}
\item $R\cong \C_q\times \D_{2r}$, $N\cong \C_q$, $\overline{R}\cong \D_{2r}$ and $p=2$. \label{case3}
\end{enumerate}

In case (\ref{case1}), we have $G=E\rtimes R\cong \C_r^x\rtimes (\C_q\rtimes \C_r)$. Since $\C_q\rtimes \C_r$ is nonabelian and acts faithfully on $\C_r^x$, we have $x\geq 2$ and $E_1\neq 1$. Since $G_1$ is not normal in $G$, we have $G_1\neq E$, hence $|EG_1:G_1|=|EG_1:E|=r$. It follows that both  $E_1$ and $G_1$ are normal in $EG_1$. Note that $EG_1$ is a maximal subgroup of $G$ and neither $E_1$ nor $G_1$ is normal in $G$ (since $G_1$ is core-free in $G$), so $\N_G(E_1)=\N_G(G_1)= EG_1$. Note that $REG_1=G$, hence 
\begin{equation}\label{eq}\tag{$\star$}
|G|=\frac{|R||EG_1|}{|R\cap EG_1|}=\frac{|R||E||G_1|}{|R\cap EG_1||E\cap G_1|}
\end{equation}
 and $|R\cap EG_1|=|G_1:E\cap G_1|=r$. Let $K=R\cap EG_1=\N_R(E_1)=\N_R(G_1)\cong \C_r$. By definition, $EK\leq EG_1$ hence $EK=EG_1$ by order considerations.

We show that, for every $s\in R\setminus K$, we have $G_1K,G_1K^s\subseteq G_1G_1^s$.  Since $K=\N_R(E_1)$, we have  $E_1^s\neq E_1$. Since $E_1$ is maximal in $E$ which is normal in $G$, it follows that $E=E_1E_1^s\subseteq G_1G_1^s$, so $G_1G_1^s=EG_1G_1^s$. Now, $G_1K\subseteq EG_1 \subseteq EG_1G_1^s=G_1G_1^s$. On the other hand, since $EK=EG_1$, we have $EG_1^s=EK^s$ and thus  $G_1K^s\subseteq EG_1K^s=EKK^s=EKEK^s=EG_1EG_1^s=G_1G_1^s$, as required. It follows by Lemma~\ref{GenWreathCosets} that $\Cay(R,S)$ is a nontrivial wreath product with respect to $K$ and the claim follows by Corollary~\ref{pq-wreath}.

In case (\ref{case2}), we have $G=E\rtimes R\cong \C_2^x\rtimes (\C_q\times \D_{2r})$. We consider faithful irreducible representations of $\C_q\times \D_{2r}$ over $\FF_2$. Since $\C_q\times \D_{2r}$ is a direct product, its representations arise as tensor products of the ones for $\C_q$ and $\D_{2r}$. Note that the faithful irreducible representations of the factors have dimension at least $2$.

Since $G_1$ is not normal in $G$,  we have  $E\neq G_1$ hence $|EG_1:E|=|EG_1:G_1|=2$ and, in particular, $EG_1\leq \N_G(E_1)$. Now, suppose $EG_1< \N_G(E_1)$, so an element of $R$ of order $q$ or $r$ normalises $E_1\cong \C_2^{x-1}$. By Maschke's Theorem, it must also normalise some $\C_2\leq E$, but this contradicts the dimensions of the faithful irreducible representations in  the previous paragraph. We conclude that $\N_G(E_1)=EG_1$. A calculation similar to (\ref{eq}) yields that $|R\cap EG_1|=|G_1:E\cap G_1|=2$. Let $K=R\cap EG_1=\N_R(E_1)\leq \N_R(G_1)$. Let $r\in R\setminus K$, so $E_1^r\neq E_1$. Since $E_1$ is normal and maximal in $E$, which itself is normal in $G$, it follows that $E_1E_1^r=E$. It follows that $G_1K\subseteq  EG_1 \subseteq G_1G_1^r$ and, by Lemma~\ref{GenWreathCosets}, $\Cay(R,S)$ is a nontrivial wreath product with respect to $K$. Note that $K\cong \C_2$, so Proposition~\ref{qDrWreath} completes the proof.

In case (\ref{case3}), $R\leq C$, so $G=ER\leq C$ and $N$ is central in $G$ hence $G\cong \C_q \times (\C_2^x\rtimes \D_{2r})$. Let $X\cong \C_2^x$ be the Sylow $2$-subgroup of $E$. Note that $X$ is normal in $G$. Since $\D_{2r}$ is nonabelian and acts faithfully on $X$, we have $x\geq 2$ and $X_1\neq 1$. Since $G_1$ is not normal in $G$,  we have  $G_1\neq X$, hence $|XG_1:G_1|=|XG_1:X|=2$. It follows that $X_1$ and $G_1$ are both normal in $XG_1$. Clearly, $\C_q$ also normalises $X_1$ and $G_1$, so $\N_G(X_1)$ and $\N_G(G_1)$ both contain $\C_q\times (X\rtimes \C_2)$ which is a maximal subgroup of $G$. Since neither $X_1$ nor $G_1$ is normal in $G$ (since $G_1$ is core-free in $G$), we have $\N_G(X_1)=\N_G(G_1)= \C_q\times (X\rtimes \C_2)$. Let $K=R\cap XG_1$ and $H=\N_R(X_1)=\N_R(G_1)$. We have $K\leq H$ and a calculation similar to (\ref{eq}) gives $|K|=2$ and $|H|=2q$ hence $K\cong \C_2$ and $H\cong \C_q\times \C_2$. Let $r\in R\setminus H$, so $X_1^r\neq X_1$. Since $X_1$ is normal and maximal in $X$, which itself is normal in $G$, it follows that $X_1X_1^r=X$. This implies that $G_1K\subseteq  XG_1 \subseteq G_1G_1^r$ and, by Lemma~\ref{GenWreathCosets}, $\Cay(R,S)$ is a nontrivial generalised wreath product with respect to $K$ and $H$, so Proposition~\ref{qDrWreath} completes the proof.

\medskip

{\noindent\textbf{$N$ is isomorphic to $\D_{2r}$:}}
In this case, $\overline{R}\cong\C_q$ and $\overline{G}\cong \C_p^x\rtimes \C_q$. Let $C$ be the centraliser of $N$ in $G$.  Note that $C\cap N=1$, so $CN=C\times N$. Suppose first that $G>CN$. Conjugation induces a natural map from $G$ to $\Aut(N)$. There is also a natural map from $\Aut(N)$ to $\Out(N)$. The kernel of the composition of these two maps is $CN$, so by the first isomorphism theorem, $G/CN$ embeds in $\Out(N)$ and $\Out(N)>1$. The only possibility is $r=5$ and $|G:CN|=2$, but this implies that $\overline{G}$ must have a subgroup of index $2$ (namely $CN/N$), which implies $q=2$, a contradiction.  We can thus assume that  $G=C\times N$ hence $\overline{G}\cong C$ and $G\cong (\C_p^x\rtimes \C_q)\times \D_{2r}$, with $p\neq q$. Since $|G_1|=p^x$ and $G_1$ is not normal in $G$, we must have $p\in\{2,r\}$. By the claim proved earlier, we can assume $p=2$ and $G\cong (\C_2^x\rtimes \C_q)\times \D_{2r}$. Let $X\cong \C_2^x$ be the Sylow $2$-subgroup of $C$. Note that $X$ is normal in $G$. Moreover, since $\C_q$ acts faithfully on $X$ and $q\geq 3$, we have $x\geq 2$. Since $G_1$ is not normal, $G_1\neq X$ hence $|XG_1:X|=|XG_1:G_1|=2$ and $X_1\neq 1$. The same calculation as in (\ref{eq}) gives $|R\cap XG_1|=2$. Write $\langle k\rangle=R\cap XG_1$. Note that $XG_1$ is elementary abelian and $XG_1=G_1\times \langle k\rangle=X\times \langle k\rangle$. Write $X=X_1\times \langle y\rangle$. Since $XG_1/X_1$ is a Klein group, there are three subgroups strictly between $X_1$ and $XG_1$, namely $X$, $G_1$ and $X_1\times \langle k\rangle$.  As $X_1\times \langle yk\rangle$ is one of these three subgroups, by elimination, we must have $X_1\times \langle yk\rangle=G_1$. Let $h$ and $m$ be elements of order $r$ and $q$ in $R$, respectively. Note that  $m$ is central in $R$ while $h^k=h^{-1}$ hence $k^h=kh^2$. Note also that $k$ and $h$ commute with $X$. We have
\begin{equation*}
(G_1)^{h^{-i}}=(X_1\langle yk\rangle)^{h^{-i}}=X_1\langle yk\rangle^{h^{-i}}=X_1\langle ykh^{-2i}\rangle.
\end{equation*}


Let $\alpha:R\to R$ be given by $(m^jh^ik^\epsilon)^\alpha=m^{-j}h^ik^\epsilon$. Note that $\alpha\in\Aut(R)$. Moreover, $\alpha\neq 1$ since $q\geq 3$.  We show that $S^\alpha=S$. Note that $\alpha$ fixes $\langle h,k\rangle$ pointwise. Let $s\in S\setminus \langle h,k\rangle$, say $s=m^jh^ik^\epsilon$, with $m^j\neq 1$.  If $\epsilon=1$, then $s^{\alpha}=m^{-j}h^ik=s^{-1}\in S^{-1}=S$. We now assume that $\epsilon=0$ so $s=m^jh^i$. Since $\langle m\rangle\cong \C_q$ acts irreducibly on $X$ and $X_1\neq 1$, we have $(X_1)^{m^j}\neq X_1$, which implies that $X_1(X_1)^{m^j}=X$ and thus $G_1(X_1)^{m^j}=G_1X_1(X_1)^{m^j}=G_1X$. Since $m^j\neq 1$, we have

\begin{align*}
G_1sG_1&=G_1m^jh^i G_1\\
&=G_1(G_1)^{h^{-i}m^{-j}}m^jh^i\\
&=G_1(X_1\langle ykh^{-2i}\rangle)^{m^{-j}}m^jh^i\\
&=G_1(X_1)^{m^{-j}}\langle ykh^{-2i}\rangle^{m^{-j}}m^jh^i\\
&=G_1X\langle ykh^{-2i}\rangle^{m^{-j}}m^jh^i\\
&=G_1\{m^jh^i,km^jh^i,kh^{-2i}m^jh^i,kkh^{-2i}m^jh^i\}\\
&=G_1\{m^jh^i,m^jh^{-i}k,m^jh^ik,m^jh^{-i}\}
\end{align*}

Since $S$ is preserved under $G_1$, we have $m^jh^{-i}\in S$ and $s^\alpha=m^{-j}h^i=(m^jh^{-i})^{-1}\in S^{-1}=S$. This completes the proof that $S^\alpha=S$ hence $\alpha\in\Aut(R)_S>1$.
\end{proof}

We can now completely determine the DRR and GRR-detecting status of these final two families of groups we have been studying.

\begin{corollary}\label{cor:special}
Let $q$ and $r$ be distinct primes.
\begin{enumerate}
\item If $R\cong \C_7\rtimes \C_3$, then $R$ is not DRR-detecting but it is GRR-detecting. \label{case1Final}
\item If $R\cong \C_q\rtimes \C_r$, with $(q,r)=(31,5)$ or $(q,r)$ a safe/Sophie Germain prime pair, with $q\equiv 3\pmod 4$ and $q\geq 11$,  then $R$ is not GRR-detecting (so is not DRR-detecting). \label{case2Final}
\item If $R$ is nonabelian and isomorphic to $\C_q\rtimes \C_r$ but not in the above two cases, then $R$ is DRR-detecting (and is therefore also GRR-detecting). \label{case3Final}
\item If $q$ is odd,  $r\in \{3,5\}$ and $R\cong \C_q\times \D_{2r}$, then $R$ is GRR-detecting. \label{case4Final}
\end{enumerate}
\end{corollary}
\begin{proof}
The statement in  (\ref{case1Final}) can be checked by computer. 

In \cite[Lemma 3.3]{PraegerXu}, it is shown that there are Cayley graphs on $\C_{31}\rtimes \C_5$ with automorphism  group $\PGaL(5,2)$. Note that $\C_{31}\rtimes \C_5$ is self-normalising (even maximal) in $\PGaL(5,2)$. 

In \cite[Lemma 4.4]{PraegerXu}, it is shown that if $q\geq 11$ is a prime (the hypothesis that $q\geq 11$ is in the paragraph before the statement of the lemma) with $q\equiv 3\pmod 4$, then there are Cayley graphs on $\C_q\rtimes \C_{(q-1)/2}$ with automorphism  group $\PSL(2,q)$. Note that $\C_q\rtimes \C_{(q-1)/2}$ is self-normalising (even maximal) in $\PSL(2,q)$. Together with the previous paragraph, this gives (\ref{case2Final}).

It remains to show (\ref{case3Final}) and (\ref{case4Final}). Let $R$ be one of the groups appearing in (\ref{case3Final}) or (\ref{case4Final}). As in Theorem~\ref{theo:SpecialCases}, let $S\subseteq R$, suppose that $S=S^{-1}$ when $R\cong \C_q\times \D_{2r}$ and let $G$ satisfy $\hat{R}<G\leq\Aut(\Cay(R,S))$, with $\hat{R}$ maximal in $G$. By Theorem~\ref{theo:SpecialCases}, we can assume the following:
\begin{itemize}
\item $\hat{R}$ is core-free in $G$ and $G$ is almost simple, or 
\item $R\cong \C_q\times \D_{2r}$, $\C_q$ is the core of $\hat{R}$ in $G$, and $G/\C_q$ is almost simple.
\end{itemize}
As in the proof of Theorem~\ref{theo:SpecialCases}, we identify $\hat{R}$ with $R$. Let $N$ be the core of $R$ in $G$, let $\overline{G}=G/N$, $\overline{R}=R/N$ and $\overline{G_1}=G_1N/N$. Note that $\overline{G}$ is an almost simple group with a maximal core-free subgroup $\overline{R}$ and another subgroup $\overline{G_1}$ such that $\overline{G}=\overline{R}\overline{G_1}$ and $\overline{R}\cap \overline{G_1}=1$. We can then view $\overline{G}$ as a primitive group of almost simple type with point-stabiliser $\overline{R}$ having a regular subgroup $\overline{G_1}$. Such groups were classified by Liebeck, Praeger, Saxl in \cite[Theorem 1.1 and Tables 16.1-16.3]{LPS}. 

When consulting these tables, it is important to remember that our point of view (for the moment) is in some sense ``dual'' to theirs:  $\overline{R}$ is our point-stabiliser so it corresponds to their $G_\alpha$. The next thing to note is that they do not list all the almost simple groups, but rather just their socles (which they denote $L$ and we will denote $\overline{L}$), and do not give $G_\alpha$, but rather $G_\alpha\cap L$. Now, $\overline{R}$ has the property that its order is squarefree, a product of at most three primes. This property is clearly preserved under subgroups, hence if $G_\alpha$ has this property, so does $G_\alpha\cap L$. So we can go through their tables and list all such instances. This is the result:

\begin{center}
\begin{tabular}{|c|c|c|c|}
\hline
$\overline{L}$ & $\overline{R}\cap \overline{L}$& Remark\\ \hline
 $\Alt(p)$& $\C_{p}\rtimes \C_{(p-1)/2}$   & $p\equiv 3\pmod 4$, $p\neq 7, 11, 23$\\  
 $\PSL(2,p)$ & $\C_{p}\rtimes \C_{(p-1)/2}$  &** \\ 
 $\PSL(2,11)$ & $\C_{11}\rtimes \C_{5}$  &\\  
  $\PSL(2,23)$ & $\C_{23}\rtimes \C_{11}$  &\\ 
 $\PSL(2,59)$& $\C_{59}\rtimes \C_{29}$   &\\ 
  $\PSL(3,3)$& $\C_{13}\rtimes \C_{3}$    &\\  
 $\PSL(3,4)$ & $\C_{7}\rtimes \C_{3}$  & *,$\overline{G}\geq \overline{L}.\Sym(3)$\\   
 $\PSL(5,2)$ & $\C_{31}\rtimes \C_{5}$&\\ 
  $\PSU(3,8)$& $\C_{19}\rtimes \C_{3}$  & *,$\overline{G}\geq \overline{L}.3^2$\\ 
  $\M_{23}$& $\C_{23}\rtimes \C_{11}$ &\\   
 $\M_{23}$& $\C_{23}\rtimes \C_{11}$  &\\ \hline
\end{tabular}
\end{center}

** In the corresponding line of \cite[Table 16.1]{LPS}, there is a remark that this case does not always occur.

Assume first that $R\cong \C_q\times \D_{2r}$, with $r\in\{3,5\}$. By Theorem~\ref{theo:SpecialCases}, $\overline{R}$ is one of $\C_q\times \D_{2r}$ or $\D_{2r}$. From the table above, we see that  $\overline{R}\cap \overline{L}$ is centreless, so either way we must have $\overline{R}\cap \overline{L}=\D_{2r}$. Again from the table above, the only case where this could occur is in the second line with $p=5$, but then we must have $\overline{G}=\overline{L}\cong \PSL(2,5)$, and one can check that there is no subgroup $\overline{G_1}$ of order $6$ in $\PSL(2,5)$ such that $\PSL(2,5)=\overline{G_1}\D_{10}$.

From now on, we assume that $R=\C_q\rtimes \C_r$ and $R$ is core-free, so $G=\overline{G}$, $R=\overline{R}$ and $L=\overline{L}$. Since $|R|$ has two prime divisors and, in the table, $R\cap L$ has at least two prime divisors, we must have $R=R\cap L$ and it follows (given the ``dual" point of view of~\cite{LPS}) that $G=L$, so $G$ is simple. This allows us to eliminate the cases which have a remark indicating that $\overline{G}>\overline{L}$, noted * in the table. (That is, we can eliminate the cases where $\overline{L}=\PSL(3,4)$ and  $\overline{L}=\PSU(3,8)$.)

Finally, we note that all remaining cases correspond to (\ref{case2Final}) of our statement (that is, $(q,r)=(31,5)$ or $(q,r)$ is a safe/Sophie Germain prime pair, with $q\equiv 3\pmod 4$ and $q\geq 11$), except the case $\overline{L}=\PSL(3,3)$, which we deal with now. According to the table, we are considering $\PSL(3,3)$ as a transitive permutation group on $13\cdot 3$ points. There are two conjugacy classes of subgroups of index $13\cdot 3$ in $\PSL(3,3)$, but they are fused in $\Aut(\PSL(3,3))$. The corresponding transitive permutation group is not primitive: it admits blocks of size $3$. This group has rank $3$ and its only non-trivial orbital digraphs are $13\K_3$ and its complement (the complete multipartite graph with $13$ parts of size $3$). It follows that $\Cay(R,S)$ is a nontrivial wreath product with respect to $\C_3$ and $\Aut(R)_S>1$ by Corollary~\ref{pq-wreath}. This concludes the  proof.
\end{proof}

Combining Theorems~\ref{main-GRR} and \ref{theo:SpecialCases}, Propositions~\ref{prop:sporadic}, \ref{prop:primeorder} and \ref{prop:abelian}, and Corollaries~\ref{cor:exceptional} and \ref{cor:special} yields Theorem~\ref{theo:main}.

\end{document}